\documentclass[a4paper,twoside,10pt]{amsart}
\usepackage[english]{babel}
\usepackage{t1enc}
\usepackage{mathptmx}
\usepackage{amsthm}
\usepackage{amssymb}
\usepackage{esint}

\newcommand{\coloneq}{\mathrel{\vcenter{\baselineskip0.5ex \lineskiplimit0pt
                     \hbox{\scriptsize.}\hbox{\scriptsize.}}}%
                     =}

\newcommand\Ccal{\mathcal C}
\newcommand\Mcal{\mathcal M}
\newcommand\Pcal{\mathcal P}
\newcommand\Nbb{\mathbb N}

\newcommand\Rbb{\mathbb R}

\newcommand\loc{\mathrm{loc}}

\newcommand\NX{{N^1\!X}}
\newcommand\NtX{{\widetilde{N}^1\!X}}
\newcommand\eps{\varepsilon}

\newcommand\Mod{\mathop{\mathrm{Mod}}\nolimits}

\newcommand\lip{\mathop{\mathrm{lip}}\nolimits}

\newcommand{\vphi}{\varphi}
\renewcommand{\theenumi}{(\alph{enumi})}

\def\itoverline#1{\skew{3}{\overline}{#1}}
\theoremstyle{plain}
\newtheorem{thm}{Theorem}[section]
\newtheorem{lem}[thm]{Lemma}
\newtheorem{pro}[thm]{Proposition}
\newtheorem{cor}[thm]{Corollary}
\theoremstyle{definition}
\newtheorem{df}[thm]{Definition}
\newtheorem{exa}[thm]{Example}
\newtheorem{rem}[thm]{Remark}
\numberwithin{equation}{section}
\title[Minimal weak upper gradients]{Minimal weak upper gradients in Newtonian spaces based~on~quasi-Banach function lattices}
\author{Luk\'{a}\v{s} Mal\'{y}}
\date{October 4, 2012}
\subjclass[2010]{Primary 46E35; Secondary 30L99, 46E30.}
\keywords{Newtonian space, upper gradient, weak upper gradient, Banach function lattice, quasi-normed space, metric measure space}
\address{Luk\'{a}\v{s} Mal\'{y}\\Department of Mathematics\\Link\"{o}ping University\\SE-581 83 Link\"{o}ping\\Sweden}
\email{lukas.maly@liu.se}
\begin{document}
\begin{abstract}
Properties of first-order Sobolev-type spaces on abstract metric measure spa\-ces, so-called Newtonian spaces, based on quasi-Banach function lattices are investigated. The set of all weak upper gradients of a Newtonian function is of particular interest. Existence of minimal weak upper gradients in this general setting is proven and corresponding representation formulae are given. Furthermore, the connection between pointwise convergence of a sequence of Newtonian functions and its convergence in norm is studied. 
\end{abstract}
\maketitle{}
%
%
\section{Introduction}
\label{sec:intro}
Generalizations of first-order Sobolev spaces in abstract metric measure spaces have been intensively studied in the past two decades. Such theories lead to new and interesting results, which can be readily used when studying functions defined on (not necessarily open) subsets of $\Rbb^n$. Shanmugalingam \cite{Sha} pioneered the theory of Newtonian spaces, corresponding to the Sobolev spaces $W^{1,p}$ for $p\in(1, \infty)$. There, the distributional gradients, which heavily rely on the linear structure of $\Rbb^n$, are substituted by the so-called upper gradients. These were originally introduced by Heinonen and Koskela \cite{HeiKos0,HeiKos}. The upper gradients are defined as Borel functions that can be used for certain pointwise estimates of differences of function values. The upper gradient of a given function is hence not determined uniquely. Therefore, the Newtonian norm is defined via a minimization process, namely,
\[
  \|u\|_\NX \coloneq \|u\|_X + \inf_g \|g\|_X\,,
\]
where $X$ is the underlying function space, e.g., $X=L^p$ corresponds to the Sobolev space $W^{1,p}$, and the infimum is taken over all upper gradients $g$ of the function $u$. 

It is only natural to ask whether this infimum is attained for some upper gradient. Since the set of upper gradients corresponding to a given function is not closed in general, one cannot really expect to obtain an affirmative answer. Indeed, the infimum need not~be attained as can be seen, e.g., in Bj\"{o}rn and Bj\"{o}rn \cite[Example~1.31]{BjoBjo} and in Mal\'{y}~\cite[Example 2.6]{Mal}. On the other hand, the same Newtonian theory can be built using weak upper gradients, which were introduced by Koskela and MacManus \cite{KoMM} as a relaxation of upper gradients. The weak upper gradients are more flexible and it might seem feasible that the infimum is attained for some weak upper gradient.

The question of existence of a unique minimal weak upper gradient has been of great interest. Shanmugalingam showed in \cite{Sha2} that a minimal weak upper gradient exists in the $L^p$ setting for $p\in(1, \infty)$. Tuominen \cite{Tuo} extended this result and proved the existence in the setting of reflexive Orlicz spaces. Mocanu \cite{Moc3} followed up the same method to show that minimal weak upper gradients exist in strictly convex reflexive Banach function spaces (see Bennett and Sharpley \cite[Definition~I.1.3]{BenSha} for the definition of Banach function spaces). In these papers, it is shown that the set of weak upper gradients of a given function is convex and closed. Then, the existence of a minimal element of this set is established using the James characteristic of reflexive spaces (see Blatter \cite{Bla}, cf. James \cite{Jam}). Even though reflexivity of the underlying function space is crucial for this method, it is not assumed, nor mentioned in \cite{Moc3}. In the present paper, we also study the properties of the sets of (weak) upper gradients in our setting. Nevertheless, we use a different approach, which does not depend on reflexivity, to find minimal weak upper gradients.

Haj\l{}asz \cite{Haj2} proved the existence of a minimal weak upper gradient in the $L^p$ setting for $p\in [1, \infty)$. He constructed a convergent sequence of weak upper gradients that minimizes a certain energy functional. The limit function was shown to be a weak upper gradient as well. 
Costea and Miranda \cite{CosMir} applied an analogous construction in the setting of the Lorentz $L^{p,q}$ spaces for $p \in (1, \infty)$ and $q \in [1, \infty)$.
We will use a similar method to prove the main theorem of the paper, i.e., that minimal weak upper gradients exist in our very general setting of quasi-Banach function lattices. Our result applies, in particular, to the $N^{1,\infty} \coloneq N^1L^\infty$ spaces, where the question of existence of minimal weak upper gradients was still open. Minimal weak upper gradients are determined uniquely pointwise up to sets of measure zero among all weak upper gradients of finite norm.

Having established the existence, we find various representation formulae for minimal weak upper gradients. Unfortunately, these do not hold in full generality (unlike the rest of the paper) since they rely on Lebesgue's differentiation theorem, which requires additional assumptions on the measure with respect to the metric, e.g., the doubling property of the measure is sufficient. The idea of representation formulae originates in Bj\"{o}rn \cite{Bjo}.

Various historical notes on the problem of minimal weak upper gradients can be found in Bj\"{o}rn and Bj\"{o}rn \cite[Section 2.11]{BjoBjo}.

The structure of the paper is the following. In Section~\ref{sec:prelim}, we give the definition of quasi-Banach function lattices, which are the underlying function spaces for this paper. We also define the Newtonian spaces and the Sobolev capacity. Section~\ref{sec:wug} provides us with an overview of the weak upper gradients and their properties that have been established in Mal\'{y} \cite{Mal} and will be used in the following text. Section~\ref{sec:minwug} is devoted to minimal weak upper gradients. Auxiliary claims as well as the main theorem are proven there. Furthermore, we find a family of representation formulae for the minimal weak upper gradients. In Section~\ref{sec:wugset}, we study the sets of (weak) upper gradients and investigate convergence properties of sequences of Newtonian functions.
%
%
\section{Preliminaries}
\label{sec:prelim}
We assume throughout the paper that $\Pcal = (\Pcal, d, \mu)$ is a metric measure space equipped with a metric $d$ and a $\sigma$-finite Borel regular measure $\mu$. In our context, Borel regularity means that all Borel sets in $\Pcal$ are $\mu$-measurable and for each $\mu$-measurable set $A$ there is a Borel set $D\supset A$ such that $\mu(D) = \mu(A)$. The connection between $d$ and $\mu$ is given by the condition that every ball in $\Pcal$ has finite positive measure. Let $\Mcal(\Pcal, \mu)$ denote the set of all extended real-valued $\mu$-measurable functions on $\Pcal$. The set of extended real numbers, i.e., $\Rbb \cup \{\pm \infty\}$, will be denoted by $\overline{\Rbb}$. The symbol $\Nbb$ will denote the set of positive integers, i.e., $\{1,2, \ldots\}$. The open ball centered at $x\in \Pcal$ with radius $r>0$ will be denoted by $B(x,r)$. We define the \emph{integral mean} of a measurable function $u$ over a measurable set $E$ of finite positive measure as
\[
  \fint_E u\,d\mu \coloneq \frac{1}{\mu(E)} \int_E u\,d\mu,
\]
whenever the integral on the right-hand side exists, not necessarily finite though.

A linear space $X = X(\Pcal, \mu)$ of equivalence classes of functions in $\Mcal(\Pcal, \mu)$ is said to be a \emph{quasi-Banach function lattice} over $(\Pcal, \mu)$ equipped with the quasi-norm $\|\cdot\|_X$ if the following axioms hold:
\begin{enumerate}
  \renewcommand{\theenumi}{(P\arabic{enumi})}
	\setcounter{enumi}{-1}
  \item \label{df:qBFL.initial} $\|\cdot\|_X$ determines the set $X$, i.e., $X = \{u\in \Mcal(\Pcal, \mu)\colon \|u\|_X < \infty\}$;
  \item \label{df:qBFL.quasinorm} $\|\cdot\|_X$ is a \emph{quasi-norm}, i.e., 
  \begin{itemize}
    \item $\|u\|_X = 0$ if and only if $u=0$ a.e.,
    \item $\|au\|_X = |a|\,\|u\|_X$ for every $a\in\Rbb$ and $u\in\Mcal(\Pcal, \mu)$,
    \item there is a constant $c\ge 1$, the so-called \emph{modulus of concavity}, such that $\|u+v\|_X \le c(\|u\|_X+\|v\|_X)$ for all $u,v \in \Mcal(\Pcal, \mu)$;
  \end{itemize}
  \item $\|\cdot\|_X$ satisfies the \emph{lattice property}, i.e., if $|u|\le|v|$ a.e., then $\|u\|_X\le\|v\|_X$;
    \label{df:BFL.latticeprop}
  \renewcommand{\theenumi}{(RF)}
  \item \label{df:qBFL.RF} $\|\cdot\|_X$ satisfies the \emph{Riesz--Fischer property}, i.e., if $u_n\ge 0$ a.e. for all $n\in\Nbb$, then $\bigl\|\sum_{n=1}^\infty u_n \bigr\|_X \le \sum_{n=1}^\infty c^n \|u_n\|_X$, where $c\ge 1$ is the modulus of concavity. Note that the function $\sum_{n=1}^\infty u_n$ needs be understood as a pointwise (a.e.) sum.
\end{enumerate}
It follows from \ref{df:qBFL.quasinorm} and \ref{df:BFL.latticeprop} that $X$ contains only functions that are finite a.e. In other words, if $\|u\|_X<\infty$, then $|u|<\infty$ a.e. A quasi-Banach function lattice is normed, and thus called a \emph{Banach function lattice}  if the modulus of concavity is equal to $1$. 

In the further text, we will slightly deviate from this rather usual definition of quasi-Banach function lattices. Namely, we will consider $X$ to be a linear space of functions defined everywhere instead of equivalence classes defined a.e. Then, the functional $\|\cdot\|_X$ is really only a quasi-seminorm.

Throughout the paper, we also assume that the quasi-norm $\|\cdot\|_X$ is \emph{continuous}, i.e., if $\|u_n - u\|_X \to 0$ as $n\to\infty$, then $\|u_n\|_X \to \|u\|_X$. The continuity of $\|\cdot\|_X$ in normed spaces follows from the triangle inequality. On the other hand, if the space $X$ is merely quasi-normed, then there is an equivalent continuous quasi-norm satisfying the lattice property due to the Aoki--Rolewicz theorem, cf. Benyamini and Lindenstrauss~\cite[Proposition H.2]{BenLin}.

The Riesz--Fischer property is equivalent to the completeness of the function lattice $X$, given that the quasi-norm is continuous and the conditions \ref{df:qBFL.initial}~--~\ref{df:BFL.latticeprop} are satisfied, cf. Zaanen \cite[Lemma 101.1]{Zaa} or Halperin and Luxemburg \cite{HalLux}.
%
%
\begin{exa}
All $L^p(\Pcal, \mu)$ spaces for $p\in[1, \infty]$ are Banach function lattices. On the other hand, if $p\in(0,1)$, then they are only quasi-Banach function lattices.

The Lorentz spaces $L^{p,r}(\Pcal, \mu)$ for $p\in (1, \infty]$ and $r\in [1, \infty]$ are Banach function lattices, where a suitable equivalent norm needs to be chosen for $p<r$. However, $L^{1,r}(\Pcal, \mu)$ for $r\in(1, \infty]$ are only quasi-Banach function lattices.

The variable exponent spaces $L^{p(\cdot)}(\Pcal, \mu)$, where $p:\Pcal\to[1,\infty]$, as well as Orlicz spaces are Banach function lattices.

Spaces of continuous, differentiable, or Sobolev functions fail to comply with the lattice property, and hence are not Banach function lattices.

A more detailed discussion on function spaces covered by this general setting can be found in Mal\'{y} \cite[Example 2.1]{Mal}. For a thorough treatise on partially ordered linear spaces, we refer the reader to Luxemburg and Zaanen \cite{LuxZaa} and Zaanen~\cite{Zaa}.
\end{exa}
%
%
\begin{df}
A sequence of measurable functions $\{f_n\}_{n=1}^\infty$ is said to \emph{converge in measure} to a measurable function $f$ on a set $M$ if for every $\eps>0$,
\[
  \mu(\{x\in M: |f(x) - f_n(x)| \ge \eps\}) \to 0 \quad \mbox{as $n\to\infty$.}
\]
\end{df}
%
%
\begin{lem}
\label{lem:conv_X-vs-measure}
Let $\{f_n\}_{n=1}^\infty$ be a sequence which converges in $X$, i.e., there is a function $f\in X$ such that $\|f_n - f\|_X \to 0$ as $n\to\infty$. Then, there is a subsequence $\{f_{n_k}\}_{k=1}^\infty$ such that $f_{n_k} \to f$ in measure on sets of finite measure as $k\to \infty$. 
\end{lem}
\begin{proof}
Let $\eps > 0$ and let $M \subset \Pcal$ be a measurable set with $\mu(M)<\infty$. Then define $E_n = \{x\in M: |f(x) - f_n(x)| \ge \eps\}$. We can estimate $\chi_{E_n}(x) \le |f(x) - f_n(x)|/\eps$ for every $x\in \Pcal$. Hence, $\|\chi_{E_n}\|_X \le \|(f-f_n)/\eps\|_X$, and consequently $\|\chi_{E_n}\|_X \to 0$ as $n\to \infty$. We want to prove that $\mu(E_{n_k}) \to 0$ as $k\to\infty$ for some subsequence, which is chosen independently of $\eps$ and $M$.

Let us choose a subsequence $\{f_{n_k}\}_{k=1}^\infty$ such that $\| f-f_{n_k} \|_X < (2c)^{-k}$, where $c\ge1$ is the modulus of concavity of $X$. Consequently, we have $\|\chi_{E_{n_k}}\| \le (2c)^{-k}/\eps $. Let $F_j = \bigcup_{k=j}^\infty E_{n_k}$. Then, we have $\mu(F_j) \le \mu(M) < \infty$ for all $j\in \Nbb$. The sets $F_j$ form a decreasing sequence. Letting $F = \bigcap_{j=1}^\infty F_j$, we obtain $\mu(F) = \lim_{j\to\infty} \mu(F_j)$. The Riesz--Fischer property gives that
\[
  \| \chi_{F_j} \|_X \le \biggl\| \sum_{k=j}^\infty \chi_{E_{n_k}} \biggr\|_X \le \sum_{k=j}^\infty c^{k+1-j} \|\chi_{E_{n_k}}\|_X \le \sum_{k=j}^\infty c^{k+1-j} \frac{1}{(2c)^k \eps} = \frac{1}{ (2c)^{j-1}\eps}
\]
for every $j\in\Nbb$. On the other hand, the functions $\chi_{F_j}$ decrease to $\chi_F$ as $j\to\infty$. The lattice property yields $\|\chi_F\|_X \le \|\chi_{F_j}\|_X \to 0$ as $j\to \infty$. Therefore, $\chi_F = 0$ a.e., whence $\mu(F) = 0$. Now, $\mu(E_{n_k}) \le \mu(F_k) \to \mu(F) = 0$ as $k\to 0$. Thus, the subsequence $\{f_{n_k}\}_{k=1}^\infty$ satisfies $\mu(\{x\in M: |f(x) - f_{n_k}(x)| \ge \eps\}) \to 0$ as $k\to\infty$ for every $\eps > 0$ and every $M\subset \Pcal$ of finite measure, i.e., it converges to $f$ in measure on sets of finite measure.
\end{proof}
%
%
\begin{lem}
\label{lem:conv_X-vs-pointwise}
Let $\{f_n\}_{n=1}^\infty$ be a sequence of measurable functions which are finite a.e. Suppose $f_n \to f$ in measure on sets of finite measure as $n\to \infty$. Then, there is a subsequence $\{f_{n_k}\}_{k=1}^\infty$ such that $f_{n_k} \to f$ a.e. in $\Pcal$. In particular, the conclusion holds if $f_n \to f$ in $X$ as $n\to\infty$.
\end{lem}
\begin{proof}
If $\mu(\Pcal)<\infty$, then the claim follows from a classical result of measure theory, see Halmos \cite[Section 22]{Hal}. Therefore, suppose $\mu(\Pcal)=\infty$. Since $\mu$ is $\sigma$-finite, we have $\Pcal = \bigcup_{n=1}^\infty \Pcal_n$, where $\mu(\Pcal_n)<\infty$ for every $n\in\Nbb$. Existence of the wanted subsequence can be established by a diagonalization argument. First we choose a subsequence $\{f_{1,k}\}_{k=1}^\infty$ of $\{f_n\}_{n=1}^\infty$ such that $f_{1,k} \to f$ a.e. on $\Pcal_1$ as $k\to \infty$. We continue inductively, i.e., having a sequence $\{f_{j,k}\}_{k=1}^\infty$ for some $j\in \Nbb$ we choose a subsequence $\{f_{j+1,k}\}_{k=1}^\infty$ such that $f_{j+1, k} \to f$ a.e. on $\Pcal_{j+1}$ as $k\to \infty$. Finally, we let $f_{n_k} = f_{k,k}$ for all $k\in\Nbb$. Hence $f_{n_k} \to f$ a.e. on $\Pcal_j$ for every $j\in \Nbb$, and therefore a.e. on $\Pcal$, as $k\to\infty$.

Suppose now that $f_n\to f$ in $X$ as $n\to \infty$. Then, there is a subsequence $\{f_{n_j}\}_{j=1}^\infty$ such that $f_{n_j} \to f$ in measure on sets of finite measure as $j\to \infty$ by Lemma \ref{lem:conv_X-vs-measure}. Now, we may choose a subsequence converging a.e. in $\Pcal$ as in the previous part of the proof.
\end{proof}
By a \emph{curve} in $\Pcal$ we will mean a rectifiable non-constant continuous mapping from a
compact interval. Thus, a curve can be (and we will always assume that all curves are) parameterized by arc length $ds$, see e.g. 
Heinonen \cite[Section~7.1]{Hei}. Note that every curve is Lipschitz continuous with respect to its arc length parametrization.

Now, we shall introduce the upper gradients, which are used as a substitute for the modulus of the usual weak gradient in the definition of our Sobolev-type norm. The upper gradients, under the name very weak gradients, were first studied by Heinonen and Koskela in \cite{HeiKos0,HeiKos}.
%
%
\begin{df}
\label{df:ug}
  Let $u: \Pcal \to \overline{\Rbb}$. A Borel function $g: \Pcal \to [0, \infty]$ is called an \emph{upper gradient} of $u$ if
\begin{equation}
 \label{eq:ug_def}
 |u(\gamma(0)) - u(\gamma(l_\gamma))| \le \int_\gamma g\,ds
\end{equation}
for all curves $\gamma: [0, l_\gamma]\to\Pcal$. To make the notation easier, we are using the convention that $|(\pm\infty)-(\pm\infty)|=\infty$.
\end{df}
%
%
\begin{df}
\label{df:DX}
A measurable function belongs to the \emph{Dirichlet space} $DX$ if it has an upper gradient in $X$.
\end{df}
%
%
\begin{df}
The \emph{Newtonian space} based on $X$ is the space $\NX = X \cap DX$ endowed with the quasi-seminorm
\begin{equation}
\label{eq:NX_norm}
  \|u\|_{\NX} = \| u \|_X + \inf_g \|g\|_X\,,
\end{equation}
where the infimum is taken over all upper gradients $g$ of $u$.
Note that the functional $\| \cdot \|_{\NX}$ can be defined by \eqref{eq:NX_norm} for every measurable function $u \notin \NX$ as well, in which case $\| u \|_{\NX} = \infty$.
Let us point out that we assume that functions are defined everywhere, and not just up to equivalence classes a.e.
\end{df}
\begin{rem}
We also define the space of natural equivalence classes $\NtX = \NX / \mathord\sim$\,, where the equivalence relation $u\sim v$ is determined by $\|u-v\|_\NX = 0$. Then, $\NtX$ is a complete (quasi)normed linear space (see Mal\'{y} \cite[Theorem 7.1]{Mal}).
\end{rem}
Note that we follow the notation of Bj\"{o}rn and Bj\"{o}rn \cite{BjoBjo} so that the symbol $\NX$ denotes the space of functions defined everywhere while $\NtX$ denotes the space of equivalence classes. Some authors, e.g., Shanmugalingam \cite{Sha,Sha2}, Tuominen \cite{Tuo},  and Mocanu~\cite{Moc3}, use the corresponding symbols the other way around.
%
%
\begin{df}
\label{df:capacity}
The \emph{(Sobolev) $X$-capacity} of a set $E\subset \Pcal$ is defined as
\[
  C_X(E) = \inf\{ \|u\|_{\NX}: u\ge 1 \mbox{ on }E\}.
\]
We say that a property of points in $\Pcal$ holds \emph{$C_X$-quasi-everywhere ($C_X$-q.e.)} if the set of exceptional points has $X$-capacity zero. Despite the dependence on $X$, we will often write simply \emph{capacity} and \emph{q.e.} whenever there is no risk of confusion of the underlying function space.
\end{df}
%
%
\begin{rem}
The capacity provides a finer distinction of sets of zero measure since $\mu(E)=0$ whenever $C_X(E) = 0$.

Moreover, $\|u\|_\NX = 0$ if and only if $u = 0$ q.e. Thus, the natural equivalence classes in $\NX$ are given by equality up to sets of capacity zero (see \cite[Corollary~6.14]{Mal}).
\end{rem}
%
%
\section{Weak upper gradients}
\label{sec:wug}
This section summarizes fundamental properties of the moduli of curve families and the weak upper gradients that have been established in the setting of Newtonian spaces based on quasi-Banach function lattices in Mal\'{y}~\cite{Mal}.
%
%
\begin{df}
\label{df:curve_families}
The family of all curves in $\Pcal$ will be denoted by $\Gamma(\Pcal)$. For an arbitrary set $E\subset \Pcal$, we define 
\[
  \Gamma_E = \{\gamma \in \Gamma(\Pcal): \gamma^{-1}(E) \neq \emptyset\}
  \quad\mbox{and}\quad
  \Gamma_E^+ = \{\gamma \in \Gamma(\Pcal): \lambda^1(\gamma^{-1}(E))> 0 \},
\]
where $\lambda^1$ denotes the (outer) 1-dimensional Lebesgue measure. 
\end{df}
\begin{rem}
If the set $\gamma^{-1}(E) \subset \Rbb$ is not $\lambda^1$-measurable, then $\lambda^1(\gamma^{-1}(E))> 0$. Observe that $\Gamma_\Pcal = \Gamma(\Pcal)$.
\end{rem}
%
%
\begin{df}
\label{df:mod}
Let $\Gamma$ be a family of curves in $\Pcal$. The \emph{$X$-modulus} of $\Gamma$ is defined by
\[
  \Mod_X(\Gamma) := \inf \| \rho\|_X\,,
\]
where the infimum is taken over all non-negative Borel functions $\rho$ satisfying $\int_\gamma \rho\,ds\ge 1$ for all $\gamma\in\Gamma$.

An assertion is said to hold for \emph{$\Mod_X$-almost every} curve (abbreviated \emph{$\Mod_X$-a.e.} curve) if the family of exceptional curves has zero $X$-modulus.
\end{df}
It is important to be able to determine whether a curve family is negligible in terms of $\Mod_X$. Hence, we use the following characterization.
%
%
\begin{pro}
\label{pro:mod0_equiv}
Let $\Gamma$ be a family of curves in $\Pcal$. Then, $\Mod_X(\Gamma) = 0$ if and only if there is a non-negative Borel function $\rho\in X$ such that $\int_\gamma \rho\,ds = \infty$ for all curves $\gamma\in\Gamma.$
\end{pro}
%
%
\begin{df}
A curve $\gamma'$ is a \emph{subcurve} of a curve $\gamma: [0, l_\gamma] \to \Pcal$ if, after re\-pa\-ra\-me\-tri\-za\-tion and perhaps reversion, $\gamma'$ is equal to $\gamma|_{[a,b]}$ for some $0\le a < b \le l_\gamma$.
\end{df}
The following lemma summarizes the fundamental properties of the $X$-modulus of a family of curves.
%
%
\begin{lem}
\label{lem:mod_properties}
The modulus satisfies the following properties given that $\Gamma_k$, $k\in\Nbb$, are families of curves in $\Pcal$.
\begin{enumerate}
  \item If $\Gamma_1 \subset \Gamma_2$, then $\Mod_X(\Gamma_1) \le \Mod_X(\Gamma_2)$.
  \item \label{lem:mod_properties:union_0}
    If $\Mod_X(\Gamma_k) = 0$ for every $k\in\Nbb$, then $\Mod_X\bigl(\bigcup_{k=1}^\infty \Gamma_k \bigr) = 0$.  
  \item \label{lem:mod_properties:subcurve_family}
    If for every curve $\gamma_1\in\Gamma_1$ there is a subcurve $\gamma_2\in\Gamma_2$ of $\gamma_1$, then $\Mod_X(\Gamma_1) \le \Mod_X(\Gamma_2)$.
\end{enumerate}
\end{lem}
The set of upper gradients of a given function lacks many useful properties, which makes the upper gradients difficult to work with. For example, upper gradients are required to be Borel, the set of upper gradients is closed neither under taking pointwise minimum of two upper gradients, nor under convergence in $X$. All these drawbacks can be resolved if we relax the conditions and introduce weak upper gradients inspired by the original idea of Koskela and MacManus \cite{KoMM}.
%
%
\begin{df}
A non-negative measurable function $g$ on $\Pcal$ is an \emph{$X$-weak upper gradient} of an extended real-valued function $u$ on $\Pcal$ if inequality \eqref{eq:ug_def} holds
for $\Mod_X$-a.e. curve $\gamma: [0, l_\gamma]\to\Pcal$.
\end{df}
%
%
\begin{df}
A Borel function $g: \Pcal\to [0, \infty]$ is called an \emph{upper gradient of $u$ along a curve $\gamma$} if it satisfies inequality \eqref{eq:ug_def} for every subcurve $\gamma'$ of $\gamma$.
\end{df}
%
%
\begin{lem}
\label{lem:not_ug_along_curves}
If $g$ is an $X$-weak upper gradient of $u$ on $\Pcal$ and
\[
  \Gamma = \{\gamma \in \Gamma(\Pcal)\colon \mbox{$g$ is not an upper gradient of $u$ along $\gamma$}\},
\]
then $\Mod_X(\Gamma) = 0$.
\end{lem}
The following lemma guarantees that the path integrals in the definition of $X$-weak upper gradient are well defined for $\Mod_X$-a.e. curve. Moreover, it shows that we may modify a function on a set of measure zero without changing the value of the corresponding path integrals over a significant number of curves.
%
%
\begin{lem}
\label{lem:mod_ae_equivalence}
Let $g_1$ and $g_2$ be non-negative measurable functions such that  $g_1=g_2$ a.e. Then, there exist Borel functions $\tilde{g}_1, \tilde{g}_2$ such that $\tilde{g}_1 = \tilde{g}_2 = g_1 = g_2$ a.e. while $\tilde{g}_1\le \min\{g_1, g_2\} \le \max\{g_1, g_2\} \le \tilde{g}_2$ everywhere. Moreover,
\[
  \int_\gamma g_1\,ds = \int_\gamma g_2\,ds = \int_\gamma \tilde{g}_1\,ds = \int_\gamma \tilde{g}_2\,ds \quad \mbox{for $\Mod_X$-a.e. curve $\gamma$.}
\]
In particular, $\int_\gamma g_1\,ds$ is well-defined for $\Mod_X$-a.e. curve $\gamma$, having a value in $[0, \infty]$.

Furthermore, if $g_1$ is an $X$-weak upper gradient of $u$, then so is $g_2$.
\end{lem}
Next, we shall see that we may approximate any $X$-weak upper gradient by an upper gradient to any desired accuracy in the $X$-norm. Therefore, the Newtonian norm~\eqref{eq:NX_norm} can be equivalently defined using weak upper gradients.
%
%
\begin{lem}
\label{lem:ug-approx-wug}
Let $g$ be an $X$-weak upper gradient of $u$. Then, there exist $\rho_k\in X$ such that $g+\rho_k$ are upper gradients of $u$ for all $k\in\Nbb$ and $\|\rho_k\|_X\to 0$ as $k\to \infty$. In fact, there is $\rho \in X$ such that we may choose $\rho_k = \rho/k$ for every $k\in\Nbb$.
\end{lem}
All functions that are equal q.e. have the same set of $X$-weak upper gradients.
%
%
\begin{lem}
\label{lem:equal_qe_same_wug}
If $u=v$ q.e. and $g$ is an $X$-weak upper gradient of $u$, then $g$ is also an $X$-weak upper gradient of $v$.
\end{lem}
The definition of an upper gradient uses the convention that $|(\pm \infty) - (\pm \infty)| = \infty$, which can make some calculations somewhat obscure. However, this convention is unnecessary when working with $X$-weak upper gradients as can be seen in the following characterization.
%
%
\begin{pro}
\label{pro:alt-def-wug}
Let $u: \Pcal \to \overline{\Rbb}$ be a function which is finite a.e. and assume that $g\ge0$ is such that for $\Mod_X$-a.e. curve $\gamma: [0, l_\gamma]\to\Pcal$ it is true that either
\[
  |u(\gamma(0))| = |u(\gamma(l_\gamma))| = \infty
  \quad \mbox{or}
  \quad
  |u(\gamma(0)) - u(\gamma(l_\gamma))| \le \int_\gamma g\,ds.
\]
Then, $g$ is an $X$-weak upper gradient of $u$.
\end{pro}
A pair of Newtonian functions related by pointwise (in)equality a.e. is actually related on a larger set, namely, quasi-everywhere.
%
%
\begin{pro}
\label{pro:NX-ineq.ae-ineq.qe}
If $u,v \in \NX$ and $u \ge v$ a.e., then $u \ge v$ q.e. In particular, if $u=v$ a.e., then $u=v$ q.e.
\end{pro}
The $X$-weak upper gradients $g\in X$ of $u\in DX$ are related to the classical derivative of $u\circ \gamma$ for every curve $\gamma$ on which the function $u$ is absolutely continuous. Note that Dirichlet functions are absolutely continuous on $\Mod_X$-a.e. curve.
%
%
\begin{lem}
\label{lem:ACC-derivative-vs-ug}
Assume that $u\in DX$ and that $g\in X$ is an $X$-weak upper gradient of $u$. Then, for $\Mod_X$-a.e. curve $\gamma: [0, l_\gamma]\to \Pcal$, we have
\[
  |(u\circ \gamma)'(t)| \le g(\gamma(t)) \quad\mbox{for a.e. }t\in[0, l_\gamma].
\]
\end{lem}
%
%
\section{Minimal weak upper gradients}
\label{sec:minwug}
In this section, we will find a distinctive $X$-weak upper gradient of a given function, which is minimal both pointwise a.e. and normwise among all $X$-weak upper gradients. The following lemmata provide us with tools that will be used in the minimization process. The method we pursue is inspired by the one used by Haj\l{}asz~\cite{Haj2}. First, we shall prove that a pointwise minimum of two $X$-weak upper gradients is again an $X$-weak upper gradient.
%
%
\begin{lem}
\label{lem:min-g1-g2}
Let $g_1, g_2 \in X$ be $X$-weak upper gradients of $u\in DX$. Then, their pointwise minimum $g\coloneq\min\{g_1, g_2\}\in X$ is an $X$-weak upper gradient of $u$, as well.
\end{lem}
\begin{proof}
We may assume that both $g_1$ and $g_2$ are Borel functions due to Lemma \ref{lem:mod_ae_equivalence}. Proposition \ref{pro:mod0_equiv} and Lemma \ref{lem:not_ug_along_curves} ensure that $\Mod_X$-a.e. curve $\gamma$ in $\Pcal$ is such that $g_1$ and $g_2$ are upper gradients of $u$ along $\gamma$, and $\int_\gamma (g_1 + g_2)\,ds<\infty$. Let now $\gamma: [0, l_\gamma]\to \Pcal$ be one such curve and let $E = \gamma^{-1} (\{x\in\Pcal: g_1(x) \le g_2(x)\})$. Then, $E\subset [0, l_\gamma]$ is a Borel set and there is a sequence of relatively open sets $U_1 \supset U_2 \supset \dots \supset E$ such that $\lambda^1(U_n \setminus E) \to 0$ as $n\to \infty$ due to the outer regularity of the Lebesgue measure $\lambda^1$ on $[0, l_\gamma]$. For a fixed $n\in\Nbb$, write $U_n$ as an at most countable union of pairwise disjoint relatively open intervals $I_i$ with endpoints $a_i<b_i$, i.e., $U_n = \bigcup_i I_i$. Then,
\begin{align*}
  |u(\gamma(0)) - u(\gamma(l_\gamma))| & \le |u(\gamma(0)) - u(\gamma(a_1))| + |u(\gamma(a_1)) - u(\gamma(b_1))| \\
  & \quad + |u(\gamma(b_1)) - u(\gamma(l_\gamma))| \\
  & \le \int_{\gamma|_{I_1}} g_1\,ds + \int_{\gamma - \gamma|_{I_1}} g_2\,ds.
\end{align*}
Splitting the interval $[0, l_\gamma]$ further with respect to the intervals $I_i$, we obtain
\[
  |u(\gamma(0)) - u(\gamma(l_\gamma))| \le \int_{\gamma|_{\bigcup_{i=1}^j I_i}} g_1\,ds + \int_{\gamma - \gamma|{}_{\bigcup_{i=1}^j I_i}} g_2\,ds, \quad\mbox{whenever $j\in\Nbb$.}
\]
Therefore,
\[
  |u(\gamma(0)) - u(\gamma(l_\gamma))| \le \int_{\gamma|_{U_n}} g_1\,ds + \int_{\gamma - \gamma|_{U_n}} g_2\,ds,
\]
where the potential passing to a limit as $j\to \infty$ is justified by the monotone and dominated convergence theorems, respectively. Applying these theorems again and letting $n\to\infty$ yield
\[
  |u(\gamma(0)) - u(\gamma(l_\gamma))| \le \int_{\gamma|_E} g_1\,ds + \int_{\gamma - \gamma|_E} g_2\,ds = \int_\gamma g\,ds.
\]
The lattice property \ref{df:BFL.latticeprop} ensures that $g\in X$.
\end{proof}
%
%
\begin{rem}
The assumption $g_1, g_2 \in X$ in the previous lemma is used when applying Proposition \ref{pro:mod0_equiv} to obtain $\int_\gamma (g_1 + g_2)\,ds<\infty$ which allows us to use the dominated convergence theorem later on. This assumption cannot be omitted as can be seen from the following example.
\end{rem}
\begin{exa}
\label{exa:cantor_infty_ug}
Let $A\subset[0,1]$ be a Borel set such that $0< \lambda^1(A \cap I) < \lambda^1(I)$ for every non-trivial interval $I\subset[0,1]$. Then, both $g_1 = \infty \chi_A$ and $g_2 = \infty \chi_{[0,1]\setminus A}$ are upper gradients of any measurable function on $[0,1]$, however, their pointwise minimum is identically zero.
\end{exa}
Next, we show how a limit of a decreasing sequence of functions translates into a limit of path integrals.
%
%
\begin{lem}
\label{lem:fuglede-type}
Let $\{g_k\}_{k=1}^\infty$ be a decreasing sequence of non-negative functions in $X$. Let $g(x) = \lim_{k\to\infty} g_k(x)$ for every $x\in\Pcal$. Then, $g\in X$ and
\[
  \int_\gamma g_k\,ds \to \int_\gamma g\,ds
\]
for $\Mod_X$-a.e. curve $\gamma$ as $k\to\infty$.
\end{lem}
\begin{proof}
The lattice property \ref{df:BFL.latticeprop} immediately implies that $g\in X$. The integral $\int_\gamma g\,ds$ is well defined and has a value in $[0, \infty)$ for $\Mod_X$-a.e. curve $\gamma$ by Lemma \ref{lem:mod_ae_equivalence} and Proposition~\ref{pro:mod0_equiv}. Let $\Gamma_\infty$ denote the exceptional family of curves for which the integral $\int_\gamma g\,ds$ does not have a real value. Similarly, $\int_\gamma g_k\,ds$ is well defined and has a value in $[0, \infty)$ for all curves $\gamma$ outside of a curve family $\Gamma_k$ whose modulus is $0$ for each $k\in\Nbb$. Considering an arbitrary curve $\gamma \in \Gamma(\Pcal) \setminus (\Gamma_\infty \cup \bigcup_{k=1}^\infty \Gamma_k)$, we obtain the expected convergence of integrals by the dominated convergence theorem. Note that Lemma \ref{lem:mod_properties}\,\ref{lem:mod_properties:union_0} implies that $\Mod_X (\Gamma_\infty \cup \bigcup_{k=1}^\infty \Gamma_k) = 0$.
\end{proof}
The following proposition shows that the set of $X$-weak upper gradients is stable under taking pointwise limits of decreasing sequences of $X$-weak upper gradients of a given function.
%
%
\begin{pro}
\label{pro:decr_lim_of_wugs_is_wug}
Let $u\in DX$ and let $\{g_k\}_{k=1}^\infty$ be a decreasing sequence of non-negative functions in $X$ all of which are $X$-weak upper gradients of $u$. For $x\in \Pcal$, we define the function $g(x) = \lim_{k\to\infty} g_k(x)$. Then, $g\in X$ is an $X$-weak upper gradient of $u$.
\end{pro}
\begin{proof}
Let $\Gamma_k \subset \Gamma(\Pcal)$ be the family of exceptional curves $\gamma$ along which $g_k$ does not satisfy \eqref{eq:ug_def}, $k\in\Nbb$. Let $\widetilde{\Gamma} \subset \Gamma(\Pcal)$ be the family of curves for which
\[
  \int_\gamma g_k\,ds \nrightarrow \int_\gamma g\,ds\quad\mbox{as $k\to\infty$.}
\]
Lemmata \ref{lem:fuglede-type} and \ref{lem:mod_properties}\,\ref{lem:mod_properties:union_0} imply that $\Mod_X(\widetilde{\Gamma} \cup \bigcup_{k=1}^\infty \Gamma_k) = 0$. If we now consider a~curve $\gamma \in \Gamma(\Pcal) \setminus (\widetilde{\Gamma} \cup \bigcup_{k=1}^\infty \Gamma_k)$, then
\[
  |u(\gamma(0)) - u(\gamma(l_\gamma))| \le \lim_{k\to\infty} \int_\gamma g_k\,ds = \int_\gamma g\,ds,
\]
whence $g$ is an $X$-weak upper gradient of $u$. The lattice property \ref{df:BFL.latticeprop} of $X$ immediately yields $g\in X$.
\end{proof}
Next, we state and prove the main theorem of the paper. We shall find an $X$-weak upper gradient which minimizes the energy functional $\inf_g \|g\|_X$ from \eqref{eq:NX_norm}, i.e., from the definition of $\|\cdot\|_\NX$. We will actually find an $X$-weak upper gradient which is minimal pointwise a.e. among all $X$-weak upper gradients in $X$ and the lattice property \ref{df:BFL.latticeprop} then implies that it has the minimal norm in $X$. Note that the following theorem is considerably more general than Mocanu's result \cite[Theorem~1]{Moc3} as we do not require $X$ to be either reflexive, or strictly convex.
%
%
\begin{thm}
\label{thm:min_wug_exists}
Let $u\in DX$. Then, there is a \emph{minimal $X$-weak upper gradient} $g_u\in X$ of $u$, i.e., $g_u \le g$ a.e. for all $X$-weak upper gradients $g\in X$ of $u$. Moreover, $g_u$ is unique up to sets of measure zero.
\end{thm}
\begin{proof}
Since $\mu$ is $\sigma$-finite, we can find sets $\Pcal_n \subset \Pcal$ such that $\Pcal = \bigcup_{n=1}^\infty \Pcal_n$ and $\mu(\Pcal_n)\in(0, \infty)$ for each $n\in \Nbb$. Therefore, we can define a quasi-additive functional $J: X\to [0, \infty)$ by 
\[
  Jf = \|f\|_X + \sum_{n=1}^\infty 2^{-n} \fint_{\Pcal_n} \frac{|f|}{1+|f|}\,d\mu.
\]
Then, $\|f\|_X \le Jf \le \|f\|_X + 1$ for all functions $f\in X$. Moreover, the functional is monotone, i.e., if $f\le h$ a.e., then $Jf \le Jh$. Let $u\in DX$ be given and let
\[
  I = \inf\{Jg\colon g\in X\mbox{ is an $X$-weak upper gradient of }u\}.
\]
Then, $I<\infty$ and there is a sequence $\{g_k\}_{k=1}^\infty \subset X$ of $X$-weak upper gradients of $u$ such that $Jg_k \to I$ as $k\to \infty$. Now, we can define $h_m = \min\{g_k\colon 1\le k \le m\}$ pointwise for all $m\in\Nbb$. All functions $h_m\in X$, $m\in\Nbb$, are $X$-weak upper gradients of $u$ by applying Lemma~\ref{lem:min-g1-g2} repeatedly, therefore, $I\le Jh_m$. On the other hand, $h_m \le g_m$ everywhere in $\Pcal$. Thus, $Jh_m \le Jg_m$, whence $Jh_m \to I$ as $m\to \infty$. We can define $h(x)=\lim_{m\to\infty} h_m(x)$ for $x\in\Pcal$. Proposition \ref{pro:decr_lim_of_wugs_is_wug} yields that $h\in X$ is an $X$-weak upper gradient of $u$, as well, and thus $Jh \ge I$. However, $Jh\le \lim_{m\to\infty} Jh_m = I$, whence $Jh = I$.

It remains to prove that $h$ is the minimal $X$-weak upper gradient of $u$. Suppose there is $h' \in X$ such that $h' < h$ on a set $A \subset \Pcal_a$ of positive measure  for some $a\in \Nbb$ (among possibly others). Let $g = \min\{h , h'\}$. Then, $g\in X$ is also an $X$-weak upper gradient of $u$ by Lemma~\ref{lem:min-g1-g2}. Moreover, $g < h$ on $A$ while $g\le h$ on $\Pcal$. Consequently, $I\le Jg \le Jh = I$.

Since $\int_A g/(1+g)\,d\mu < \int_A h/(1+h)\,d\mu$, we obtain 
\[
  I = Jg = \|g\|_X + 2^{-a} \fint_{\Pcal_a} \frac{g}{1+g} + \sum_{\substack{n\in \Nbb\\ n\neq a}} 2^{-n} \fint_{\Pcal_n} \frac{g}{1+g} < Jh = I
\]
which is a contradiction. Therefore, the inequality $h'<h$ holds on a set of zero measure. In other words, $h'\ge h$ a.e. on $\Pcal$.
\end{proof}
%
%
\begin{rem}
In the further text, $g_u$ will denote the minimal $X$-weak upper gradient of $u\in DX$. We will consider $g_u$ to be defined everywhere in $\Pcal$ even though there is some freedom in choosing its representative.
As already mentioned, the minimal $X$-weak upper gradient $g_u\in X$ of $u\in DX$ satisfies $\|g_u\|_X \le \|g\|_X$ for all $X$-weak upper gradients $g\in\Mcal(\Pcal, \mu)$ of~$u$. Observe however that the pointwise inequality holds only for $X$-weak upper gradients $g\in X$ of $u$. The following example shows that it might fail otherwise.
\end{rem}
\begin{exa}
\label{exa:cantor_infty_gradient}
Suppose $X\subset L^1([0,1])$. Similarly as in Example \ref{exa:cantor_infty_ug}, let $A \subset [0, 1]$ be a Borel set which satisfies $0<\lambda^1(A \cap I)<\lambda^1(I)$ for all non-degenerate intervals $I \subset [0,1]$. Then, $g= \infty \chi_A$ is an upper gradient of any function on $[0,1]$. Let $u(x) = x$ for $x\in [0,1]$. Thus, $g_u = 1$ a.e. as proven in Lemma~\ref{lem:Lip-min_wug} below. Nevertheless, $g_u > g$ on $[0,1]\setminus A$, i.e., on a set of positive measure.
\end{exa}
Existence of a minimal $X$-weak upper gradient guarantees that we may replace the infimum in the definition of the Newtonian quasi-seminorm by the $X$-norm of the minimal $X$-weak upper gradient.
%
%
\begin{cor}
\label{cor:NX-norm_via_min-wug}
If $u\in \NX$, then $\|u\|_{\NX} = \|u\|_X + \|g_u\|_X$.
\end{cor}
\begin{proof}
By Lemma \ref{lem:ug-approx-wug}, we can find a sequence of upper gradients $\{g_j\}_{j=1}^\infty$ of $u$ such that $g_j \to g_u$ in $X$ as $j\to\infty$. Therefore,
\[
  \|u\|_\NX \le \|u\|_X + \lim_{j\to \infty} \|g_j\|_X = \|u\|_X + \|g_u\|_X\,.
\]
On the other hand, $\|g_u\|_X \le \|g\|_X$ for every upper gradient $g$ of $u$. Hence,
\[
  \|u\|_\NX = \|u\|_X + \inf_g \|g\|_X \ge \|u\|_X + \|g_u\|_X\,,
\]
where the infimum is taken over all upper gradients $g$ of $u$.
\end{proof}
We shall see that it is possible to find some representation formulae for minimal $X$-weak upper gradients. The proof relies heavily on Lebesgue's differentiation theorem, which reads
\[
  \lim_{r \to 0} \fint_{B(x,r)} f\,d\mu = f(x) \quad \mbox{for a.e. }x\in \Pcal
\]
whenever $f\in L^1_\loc(\Pcal)$. Lebesgue's differentiation theorem is known to hold when $\mu$ is a doubling measure, i.e., when there is $C\ge 1$ such that $\mu(2B) \le C \mu(B)$ for every ball $B\subset \Pcal$, or more generally when the Vitali covering theorem holds. For further details on Lebesgue's differentiation theorem, see Bj\"{o}rn and Bj\"{o}rn \cite[Section 2.10]{BjoBjo}.
%
%
\begin{thm}
\label{thm:min_wug_formula}
Suppose that Lebesgue's differentiation theorem holds for $(\Pcal, \mu)$. Let $\vphi\in\Ccal([0, \infty])$ be a strictly increasing non-negative function with $\vphi(0)=0$  such that $\vphi\circ |f| \in L^1_\loc(\Pcal)$ whenever $f\in X$. Suppose further that one of the following two conditions is fulfilled:
\begin{enumerate}
  \item $\vphi$ is subadditive, i.e., $\vphi(x+y) \le \vphi(x) + \vphi(y)$ for any $x,y\in[0, \infty];$
  \item \label{itm:minkowski-type}
  $\vphi$ supports a Minkowski-type inequality, i.e., the inequality
  \[
    \vphi^{-1}\Bigl( \fint_E \vphi \circ (f + h)\,d\mu\Bigr) \le \vphi^{-1}\Bigl( \fint_E \vphi \circ f\,d\mu\Bigr)  + \vphi^{-1}\Bigl( \fint_E \vphi \circ h \,d\mu\Bigr) 
  \]
  holds for any measurable $E\subset \Pcal$ of finite measure, and any non-negative measurable functions $f$ and $h$.
\end{enumerate}
Let $u\in DX$ and define
\[
  g_1(x)  = \inf_g \limsup_{r\to 0^+} \vphi^{-1} \Bigl(\fint_{B(x,r)} \vphi\circ g\,d\mu\Bigr), \quad x\in\Pcal,
\]
where the infimum is taken over all $X$-weak upper gradients $g\in X$ of $u$. Let further $g_2$ be defined similarly but taking the infimum only over all upper gradients $g\in X$ of $u$.
Then, $g_1 = g_2 = g_u$ a.e., and thus both $g_1$ and $g_2$ are minimal $X$-weak upper gradients of $u$.
\end{thm}
\begin{proof}
Note that $\vphi^{-1}$ is a continuous function because $\vphi$ is assumed continuous on the compactified interval $[0, \infty]$. The function $\vphi\circ g$ is locally integrable for every non-negative $g\in X$. Thus, Lebesgue's differentiation theorem applies. Let
\[
  g_u^*(x) = \limsup_{r\to 0^+} \vphi^{-1} \Bigl( \fint_{B(x,r)} \vphi\circ g_u\,d\mu\Bigr), \quad x\in\Pcal.
\]
Then, $\vphi \circ g_u^* = \vphi\circ g_u$ a.e. On the other hand, we have $\vphi\circ g_u \le \vphi \circ g$ a.e. for every $X$-weak upper gradient $g\in X$ of $u$, and hence the integral means of $\vphi\circ g_u$ are dominated by the integral means of $\vphi\circ g$. Thus, $g_u^* \le g_1 \le g_2$ a.e.
Due to Lemma \ref{lem:ug-approx-wug}, there is a non-negative Borel function $\rho\in X$ such that all functions $g_u + \rho/j$, $j\in \Nbb$, are upper gradients of $u$. Let 
\[
  M = \{ x\in \Pcal: \rho(x) < \infty \} \cap \bigcap_{j=1}^\infty \biggl\{x\in \Pcal: \mbox{ $x$ is a Lebesgue point of $\vphi\biggl(\dfrac{\rho(\cdot)}{j}\biggr)$ }\biggr\}.
\]
Then, $\mu(\Pcal \setminus M) = 0$ as $\rho\in X$ is finite a.e. and Lebesgue's differentiation theorem applies to all functions $\vphi\Bigl(\tfrac{\rho(\cdot)}{j}\Bigr)$, $j\in\Nbb$, since they are locally integrable.

Suppose first that $\vphi$ is subadditive, which gives for $x\in M$ that
\begin{align*}
  \vphi(g_2(x)) & \le \limsup_{r\to 0^+} \fint_{B(x,r)} \vphi\biggl(g_u(y) + \frac{\rho(y)}{j}\biggr)\,d\mu(y) \\
  & \le \limsup_{r\to 0^+} \fint_{B(x,r)} \vphi(g_u(y))\,d\mu(y) + \limsup_{r\to 0^+} \fint_{B(x,r)} \vphi\biggl(\frac{\rho(y)}{j}\biggr)\,d\mu(y) \\
  &= \vphi(g_u^*(x)) + \vphi\biggl( \frac{\rho(x)}{j} \biggr).
\end{align*}
Letting $j\to\infty$ shows that $\vphi\circ g_2 \le \vphi\circ g_u^*$ on $M$, and hence $g_2 \le  g_u^*$ a.e. on $\Pcal$.

Suppose now that $\vphi$ supports the Minkowski-type inequality. For $x\in M$, we have
\begin{align*}
  g_2(x) & \le \limsup_{r\to 0^+} \vphi^{-1} \biggl(\fint_{B(x,r)} \vphi\biggl(g_u(y) + \frac{\rho(y)}{j}\biggr)\,d\mu(y)\biggr) \\
  & \le \limsup_{r\to 0^+} \vphi^{-1} \biggl(\fint_{B(x,r)} \vphi(g_u(y))\,d\mu(y)\biggr) \\
  & \quad + \limsup_{r\to 0^+} \vphi^{-1}\biggl(\fint_{B(x,r)} \vphi\biggl(\frac{\rho(y)}{j}\biggr)\,d\mu(y)\biggr) \\
  &= g_u^*(x) + \frac{\rho(x)}{j}\,.
\end{align*}
Letting $j\to\infty$ shows that $g_2 \le g_u^*$ on $M$, i.e., almost everywhere on $\Pcal$.
\end{proof}
\begin{exa}
Let us find several examples of functions $\vphi$, which satisfy the hypotheses of the previous theorem.
\begin{enumerate}
	\item The representation formula based on either $\vphi(t) = t/(1+t)$ or $\vphi(t) = \arctan t$, where $t\in [0, \infty)$, may be used for any (quasi)Banach function lattice $X$ as both functions are bounded. Therefore, $\vphi \circ g \in L^\infty(\Pcal) \subset L^1_\loc(\Pcal)$ for any measurable function $g$. These functions $\vphi$ are concave, and hence subadditive.
	\item \label{exa:repre_form_phi_b}
	Given $p\in (0, 1]$, the function $\vphi(t) = t^p$ for $t\in[0, \infty)$ is concave (and hence subadditive), but unbounded, which somewhat restricts the choice of the function space~$X$. For example, let $X=L^q(\Pcal)$. If $q\in[p, \infty]$, then the assumptions are satisfied. On the other hand, if $q\in(0, p)$, then $\vphi \circ g$ in general fails to be locally integrable for some $g\in X$.
  \item Given $p\in [1, \infty)$, the function $\vphi(t) = t^p$, $t\in[0, \infty)$, obeys the Minkowski-type inequality, which in this case is actually just the triangle inequality for the $L^p$ norm. Similarly as in \ref{exa:repre_form_phi_b}, the theorem's hypotheses are not fulfilled in the case of $X=L^q(\Pcal)$ with $q\in(0, p)$.
  \item Mulholland \cite{Mul} has shown that a function $\vphi$ satisfies the Minkowski-type inequality if $\log(\vphi'(e^t))$ is an increasing and concave function for $t\in \Rbb$. This condition can be equivalently written as $\vphi(t) = \int_0^t e^{\psi(\log \tau)}\,d\tau$, where $\psi: \Rbb\to\Rbb$ is increasing and concave. Consequently, one can show that the function $\vphi(t) = t^p (\log (1+t))^q$ satisfies condition \ref{itm:minkowski-type} in the theorem's hypotheses whenever $p,q\in[1,\infty)$.
  \item Matkowski \cite{Mat} has proven that the Minkowski-type inequality in our setting holds if and only if the function $F(s,t) = \vphi(\vphi^{-1}(s) + \vphi^{-1}(t))$, $s,t\ge 0$, is concave on $[0, \infty) \times [0, \infty)$. He also generalized Mulholland's sufficient condition. If $\vphi \in \Ccal^2((0,\infty))$ is strictly convex and the function $\vphi'/\vphi''$ is superadditive in $(0, \infty)$, then $F(s,t)$ is indeed concave, and thus the Minkowski-type inequality holds. Based on this condition, one can prove that the functions $\vphi(t) = t^2 / (t+1)$ and $\vphi(t) = t e^{-1/t}$, $t>0$, may be used for  suitable function spaces $X$.
\end{enumerate}
\end{exa}
%
%
\section{The set of weak upper gradients}
\label{sec:wugset}
In this section, we will study convergence properties of sequences of $X$-weak upper gradients. A fundamental result has been established by Fuglede~\cite{Fug} in the setting of Lebesgue spaces on $\Rbb^n$. Later it was generalized by Shanmugalingam~\cite{Sha} to Lebesgue spaces on metric measure spaces, and it turns out that the lemma holds true even if we choose quasi-Banach function lattices to be the underlying function spaces.
%
%
\begin{lem}[Fuglede's lemma]
\label{lem:fuglede}
Assume that $g_k\to g$ in $X$ as $k\to\infty$. Then, there is a subsequence (again denoted by $\{g_k\}_{k=1}^\infty$) such that
\[
  \int_\gamma g_k\,ds \to \int_\gamma g\,ds \quad \mbox{as }k\to\infty
\]
for $\Mod_X$-a.e. curve $\gamma$, while all the integrals are well defined and real-valued. Furthermore, for $\Mod_X$-a.e. curve $\gamma$,
\[
  \int_\gamma |g_k-g|\,ds \to 0 \quad \mbox{as }k\to\infty.
\]
\end{lem}
\begin{rem}
A similar claim has been proven in Lemma~\ref{lem:fuglede-type}. The hypothesis of convergence of the sequence $\{g_k\}_{k=1}^\infty$ in $X$ is replaced there by its monotone pointwise convergence. Note that these lemmata are not corollaries of each other, as can be seen from the following examples. 
\begin{enumerate}
  \item Suppose $X=L^\infty([0,1])$ and let $g_k = \chi_{(0, 1/k)}$ for $k\in\Nbb$. Then, $g_k$ decreases to zero function as $k\to\infty$. However, neither this sequence nor any of its subsequences converges in $L^\infty([0,1])$, as it is not a Cauchy sequence, which is a key property that will be used in the proof below.
  \item Suppose $X=L^1([0, 1])$ and let $g_k = k \chi_{(0, 1/k)} / \log k$ for $k>1$. Then, $g_k \to 0$ in $L^1([0, 1])$ as $k\to\infty$. However, the pointwise convergence is not monotone. Moreover, the argument in the proof of Lemma~\ref{lem:fuglede-type} was based on the dominated convergence theorem, which would fail here since no dominating function would be integrable.
\end{enumerate}
\end{rem}
In the following proof, the symbols $g^+$ and $g^-$ are used to denote the positive and the negative part of a function $g$, respectively, i.e.,
\[
  g^+ = \max\{0, g\} \quad \mbox{and}\quad g^- = (-g)^+ = \max\{0, -g\}.
\]
\begin{proof}[Proof of Lemma~\ref{lem:fuglede}]
Passing to a subsequence, we may assume that $\|g_k - g\|_X < (2c)^{-k}$, where $c\ge1$ is the modulus of concavity of $X$. By Lemma \ref{lem:mod_ae_equivalence}, $\int_\gamma g^+\,ds$ is well defined with a value in $[0, \infty]$ for $\Mod_X$-a.e. curve $\gamma$. Furthermore, $\int_\gamma g^+\,ds < \infty$ for $\Mod_X$-a.e. curve $\gamma$ by Proposition~\ref{pro:mod0_equiv}. Similarly, $\int_\gamma g^-\,ds$ is well defined and real-valued for $\Mod_X$-a.e. curve $\gamma$. Consequently, $\int_\gamma g\,ds$ is well defined and real-valued for $\Mod_X$-a.e. curve $\gamma$. Let the family of the exceptional curves be denoted by $\Gamma_\infty$. Arguing similarly for each $k\in\Nbb$, we obtain families $\Gamma_k$ with $\Mod_X(\Gamma_k)=0$, outside of which $\int_\gamma g_k\,ds$ is well defined and real-valued. Let $\Gamma = \Gamma_\infty \cup \bigcup_{k=1}^\infty \Gamma_k$. Then, all the integrals are well defined and real-valued for $\gamma \in \Gamma(\Pcal)\setminus \Gamma$, while $\Mod_X(\Gamma)=0$ by Lemma \ref{lem:mod_properties}\,\ref{lem:mod_properties:union_0}.
  
Let next
\begin{alignat*}{2}
  \widehat{\Gamma} & = \Bigl\{ \gamma \in \Gamma(\Pcal) \setminus \Gamma: \int_\gamma g_k\,ds \not\to \int_\gamma g\,ds, \mbox{ as } k\to \infty\Bigr\}\,,\quad\\
  \widetilde{\Gamma} & = \Bigl\{ \gamma \in \Gamma(\Pcal) \setminus \Gamma: \int_\gamma |g_k-g|\,ds \not\to 0, \mbox{ as } k\to \infty\Bigr\}\,,& \\
  \widetilde{\Gamma}_j & = \Bigl\{ \gamma \in \Gamma(\Pcal) \setminus \Gamma: \limsup_{k\to\infty} \int_\gamma |g_k-g|\,ds > \frac{1}{j}\Bigr\}\,, & j\in\Nbb.
\end{alignat*}
Then, $\widehat{\Gamma} \subset \widetilde{\Gamma} = \bigcup_{j=1}^\infty \widetilde{\Gamma}_j$. Let $\rho_{m,j} = j \sum_{k=m+1}^\infty |g_k - g|$. Then,
\[
  \int_\gamma \rho_{m,j}\,ds > 1 \quad \mbox{for all $\gamma\in\widetilde{\Gamma}_j$ and $m\in\Nbb$}.
\]
Hence,
\[
  \Mod_X(\widetilde{\Gamma}_j) \le \|\rho_{m,j}\|_X \le j \sum_{k=m+1}^\infty c^{k-m} (2c)^{-k} = j (2c)^{-m} \to 0\quad\mbox{as $m\to\infty$,}
\]
which yields $\Mod_X(\Gamma\cup\widehat{\Gamma}) = \Mod_X(\Gamma\cup\widetilde{\Gamma}) = 0$ due to Lemma \ref{lem:mod_properties}\,\ref{lem:mod_properties:union_0}. Finally, for every curve $\gamma \in \Gamma(\Pcal)\setminus(\Gamma \cup \widehat{\Gamma})$ we have $\int_\gamma g_k\,ds\to\int_\gamma g\,ds$ as $k\to\infty$. Moreover, $\int_\gamma |g_k - g|\,ds \to 0$ for every curve $\gamma \in \Gamma(\Pcal)\setminus(\Gamma \cup \widetilde{\Gamma})$ as $k\to\infty$.
\end{proof}
One of the fundamental disadvantages of the set of upper gradients of a given function is that it is not closed under convergence in $X$. The following proposition shows that, on the contrary, the set of $X$-weak upper gradients is closed in $X$.
%
%
\begin{pro}
\label{pro:X-lim_of_wugs_is_wug}
Let $u\in DX$, and $\{g_k\}_{k=1}^\infty$ be a sequence of $X$-weak upper gradients of $u$ in $X$. Suppose that $g_k \to g$ in $X$ as $k\to\infty$, where $g\in X$ is non-negative. Then, $g$ is an $X$-weak upper gradient of $u$.
\end{pro}
\begin{proof}
Let $\Gamma_k \subset \Gamma(\Pcal)$ be the set of those curves for which $g_k$ does not satisfy \eqref{eq:ug_def}, $k\in\Nbb$. Then, $\Mod_X(\Gamma_k) = 0$ by the definition of $X$-weak upper gradients. Let $\Gamma$ be the family of curves for which $\int_\gamma g_k\,ds \not\to \int_\gamma g\,ds$, or some of the integrals are not well defined. Fuglede's lemma shows that $\Mod_X(\Gamma)=0$. Consider now a curve $\gamma \in \Gamma(\Pcal) \setminus(\Gamma \cup \bigcup_{k=1}^\infty \Gamma_k)$. Then,
\[
  | u(\gamma(0)) - u(\gamma(l_\gamma)) | \le \lim_{k\to\infty} \int_{\gamma} g_k\,ds = \int_\gamma g\,ds.
\]
Lemma~\ref{lem:mod_properties}\,\ref{lem:mod_properties:union_0} yields that $\Mod_X(\Gamma \cup \bigcup_{k=1}^\infty \Gamma_k)=0$, and hence $g$ is an $X$-weak upper gradient of $u$.
\end{proof}
The following proposition further studies the closedness of the set of $X$-weak upper gradients of a given function.
%
%
\begin{pro}
\label{pro:ug_set_closure}
For $u\in DX$, let $M$ be the set of all upper gradients of $u$ which belong to $X$. Then, the closure $\itoverline{M}$ of $M$ in $X_+$ consists precisely of those $X$-weak upper gradients of $u$, which are in $X$. Moreover, $\itoverline{M} = \{g\in X_+: g\ge g_u\mbox{ a.e.}\}$.
\end{pro}
In the proposition, we use the symbol $X_+$, which denotes the convex cone of non-negative functions (not equivalence classes) in $X$, equipped with the (quasi)semi\-norm inherited from $X$.
\begin{proof}
Let $g\in X$ be an $X$-weak upper gradient of $u$. Then, $g\in \itoverline{M}$ by Lemma \ref{lem:ug-approx-wug}.

Conversely, let $g\in \itoverline{M}$. Then, there exists a sequence $\{g_j\}_{j=1}^\infty \subset X_+$ of upper gradients of $u$ such that $g_j \to g$ in $X$ as $j\to \infty$. Proposition \ref{pro:X-lim_of_wugs_is_wug} shows that $g$ is an $X$-weak upper gradient of $u$.

Let $M' = \{g\in X_+: g\ge g_u\mbox{ a.e.}\}$. Then, $\itoverline{M} \subset M'$ since $g_u$ is minimal a.e. among all $X$-weak upper gradients $g\in X_+$ by Theorem \ref{thm:min_wug_exists}. On the other hand, let $g\in M'$. Then, $g \ge \min\{g, g_u\}$ everywhere in $\Pcal$ while $\min\{g, g_u\} = g_u$ a.e. in $\Pcal$. The function $\min\{g, g_u\}$ is an $X$-weak upper gradient of $u$ by Lemma \ref{lem:mod_ae_equivalence}, and hence so is $g$.
\end{proof}
Fuglede's lemma (Lemma \ref{lem:fuglede}) has an interesting consequence about convergence of sequences of Newtonian functions.
%
%
\begin{pro}
\label{pro:XxX-vs-N1X_convergence}
Let $f_k\in \NX$ and suppose that $g_k\in X$ is an $X$-weak upper gradient of $f_k$, $k\in\Nbb$. Assume further that $f_k\to f$ in $X$ and $g_k \to g$ in $X$ as $k\to\infty$, where $g$ is non-negative. Then, there is a function $\tilde{f} = f$ a.e. such that $g$ is an $X$-weak upper gradient of $\tilde{f}$, and thus $\tilde{f} \in \NX$. Furthermore, there is a subsequence $\{f_{k_j}\}_{j=1}^\infty$ such that $f_{k_j} \to \tilde{f}$ q.e. as $k\to\infty$.

If either $f\in\NX$ or there is a subsequence $\{f_{k_j}\}_{j=1}^\infty$ such that $f_{k_j} \to f$ q.e. as $k\to\infty$, then we may choose $\tilde{f} = f$.
\end{pro}
%
%
\begin{rem}
Observe that we will not prove that $f_k \to \tilde{f}$ in $\NX$ as $k\to\infty$. Such a conclusion need not be true, as can be seen from Example~\ref{exa:XxX-vs-N1X}.
\end{rem}
%
%
\begin{proof}[Proof of Proposition \ref{pro:XxX-vs-N1X_convergence}]
We may assume that $f_k \to f$ a.e., passing to a subsequence if necessary. If the assumption on existence of a subsequence converging q.e. is fulfilled, then that one is the subsequence we pass to. Otherwise, we use Lemma \ref{lem:conv_X-vs-pointwise} to justify this step. By Fuglede's lemma there is a family of curves $\Gamma$ with $\Mod_X(\Gamma)=0$ such that $\int_\gamma g_k\,ds \to \int_\gamma g\,ds \in \Rbb$ as $k\to\infty$ whenever $\gamma\in \Gamma(\Pcal)\setminus \Gamma$. Let us define $\tilde{f}$ pointwise everywhere in $\Pcal$ by $\tilde{f} = \limsup_{k\to\infty} f_k$. Then, $\tilde{f} = f$ a.e. in $\Pcal$. 

Lemmata \ref{lem:mod_properties} and \ref{lem:not_ug_along_curves} show that $g_k$ is an upper gradient of $f_k$ for all $k\in\Nbb$ along $\Mod_X$-a.e. curve $\gamma$, while neither $\gamma$, nor any of its subcurves belong to $\Gamma$. Let us now consider one such curve $\gamma: [0,l_\gamma]\to\Pcal$. Then, either $|\tilde{f}(\gamma(0))| = |\tilde{f}(\gamma(l_\gamma))| = \infty$, or 
\begin{align*}
  |\tilde{f}(\gamma(0)) - \tilde{f}(\gamma(l_\gamma))| & \le \limsup_{k\to\infty} |f_k(\gamma(0)) - f_k(\gamma(l_\gamma))| \\
  & \le \limsup_{k\to\infty} \int_\gamma g_k\,ds = \int_\gamma g\,ds.
\end{align*}
As $\tilde{f}$ is finite a.e., Proposition \ref{pro:alt-def-wug} shows that $g$ is an $X$-weak upper gradient of $\tilde{f}$.

Let now $\hat{f} = \liminf_{k\to\infty} f_k$. Then, $\hat{f}=f=\tilde{f}$ a.e. in $\Pcal$. An analogous argument as above yields that $g$ is an $X$-weak upper gradient of $\hat{f}\in \NX$, as well. By Proposition~\ref{pro:NX-ineq.ae-ineq.qe}, we obtain that $\hat{f} = \tilde{f}$ q.e., and hence $f_k \to \tilde{f}$ q.e. as $k\to \infty$.

If $f\in \NX$, then $f=\tilde{f}$ q.e. by Proposition \ref{pro:NX-ineq.ae-ineq.qe}, and $g$ is an $X$-weak upper gradient of $f$ by Lemma \ref{lem:equal_qe_same_wug}. Moreover, $f_k\to f$ q.e. as $k\to \infty$.

Finally, if $f_{k_j} \to f$ q.e. as $j\to\infty$, then again $f=\tilde{f}$ q.e., and $g$ is an $X$-weak upper gradient of $f$ by Lemma \ref{lem:equal_qe_same_wug}.
\end{proof}
The following lemma provides us with an explicit description of the minimal $X$\nobreakdash-weak upper gradient of a locally Lipschitz function defined on an interval on the real line endowed with the Lebesgue measure, given that all functions in $X$ are locally integrable. We will use the formula in Example~\ref{exa:XxX-vs-N1X} below. We also see in the proof that in such a setting, all $X$-weak upper gradients of an arbitrary measurable function are actually its upper gradients.
%
%
\begin{lem}
\label{lem:Lip-min_wug}
Assume that $X \subset L^1_\loc(I)$, where $I \subset \Rbb$ is an interval. Let $u\in DX$ be a locally Lipschitz function. Then, the \emph{(lower) pointwise dilation}
\[
  \lip u(x) = \liminf_{r\to 0} \sup_{y\in B(x,r)\cap I} \frac{|u(y) - u(x)|}{r}\,, \quad x\in I,
\]
is a minimal $X$-weak upper gradient of $u$. Moreover, it is an upper gradient of $u$.
\end{lem}
\begin{proof}
Cheeger \cite[Proposition 1.11]{Che} has proven that $\lip u$ is an upper gradient of $u$.
A classical result of Lebesgue yields that $u$ is differentiable a.e. in $I$ (cf. the Rademacher theorem in Ziemer~\cite[Theorem 2.2.1]{Zie}). Moreover, $|u'(x)| = \lip u(x)$ for a.e. $x\in I$.

Let $[a,b] \subset I$ be a bounded interval and let $\gamma_{a,b}(t) = a + t$ for $t \in [0, b-a]$. Then, the singleton curve family $\Gamma = \{\gamma_{a,b}\}$ has a positive $L^1$-modulus, and hence $\Mod_X(\Gamma)>0$ by Proposition~\ref{pro:mod0_equiv}. Let $g\in X$ be an arbitrary $X$-weak upper gradient of $u$. By Lemma~\ref{lem:ACC-derivative-vs-ug}, we have
\[
  |(u\circ \gamma_{a,b})'(t)| \le g(\gamma_{a,b}(t)) \quad\mbox{ for a.e. }t\in[0, b-a],
\]
which yields that $|u'| \le g$ a.e. on $[a,b]$. Consequently, $\lip u = |u'| \le g$ a.e. on $I$ whence $\lip u$ is a minimal $X$-weak upper gradient of $u$.
\end{proof}

%
%
\begin{exa}
\label{exa:XxX-vs-N1X}
Suppose that $X\subset L^1([-1,1])$ contains non-zero constant functions. Consider the triangle wave functions $f_k(x) = \arccos(\cos kx)/k$ with upper gradients $g_k \equiv 1$, where $x\in [-1, 1]$ and $k\in\Nbb$. Then, $f_k \to f\equiv 0$ everywhere in $[-1,1]$, while $g_k \to g\equiv 1$ both in $X$ and pointwise everywhere as $k\to\infty$. Obviously, $f \in \NX$. However, $f_k \not\to f$ in $\NX$ as $k\to\infty$. Indeed, by Lemma \ref{lem:Lip-min_wug}, $g_k$ is a minimal $X$-weak upper gradient of $f_k$, and hence
\[
  \|f-f_k\|_\NX = \|f_k\|_\NX = \|f_k\|_X + \|g_k\|_X \ge \|g\|_X>0\quad\mbox{for all $k\in\Nbb$.}
\]
\end{exa}
The following proposition resembles Proposition \ref{pro:XxX-vs-N1X_convergence}; however, we relax the assumption that $f_k \in \NX$, but the convergence of $f_k$ needs to be stronger, namely, pointwise quasi-everywhere.
%
%
\begin{pro}
\label{pro:qe+X-vs-N1X_convergence}
Let $f_k\in DX$ and suppose that $g_k\in X$ is an $X$-weak upper gradient of $f_k$, $k\in\Nbb$. Assume further that $f_k\to f$ q.e. and $g_k\to g$ in $X$ as $k\to\infty$, where $f$ is real-valued almost everywhere while $g$ is non-negative. Then, $g$ is an $X$-weak upper gradient of $f$.
\end{pro}
Observe that the assumption that $f$ is real-valued a.e. is crucial for the claim. For example, if we let $f_k \equiv k$ with $g_k \equiv 0$ for all $k\ge 1$, then $g \equiv 0$ is not an $X$-weak upper gradient of $f \equiv \infty$, given that the space $\NX$ is a proper subspace of $X$.
\begin{proof}
By passing to a subsequence if necessary we may assume by Fuglede's lemma that $\int_\gamma g_k \, ds \to \int_\gamma g\,ds\in\Rbb$ as $k\to\infty$ whenever $\gamma\in\Gamma(\Pcal)\setminus \Gamma$, where $\Mod_X(\Gamma)=0$. Let $\tilde{f} = \limsup_{k\to\infty} f_k$ and $E=\{x\in\Pcal: |\tilde{f}(x)| = \infty\}$. Then, $\tilde{f}=f$ q.e.

In the same way as in the proof of Proposition~\ref{pro:XxX-vs-N1X_convergence}, we can use Lemmata \ref{lem:mod_properties} and~\ref{lem:not_ug_along_curves} with Proposition \ref{pro:alt-def-wug} to show that $g$ is an $X$-weak upper gradient of $\tilde{f}$.

Finally, $g$ is an $X$-weak upper gradient of $f$ by Lemma \ref{lem:equal_qe_same_wug}.
\end{proof}
%
%
\begin{rem}
In Example~\ref{exa:XxX-vs-N1X}, the limit function $g$ was an upper gradient of $f$, however it was not its minimal $X$-weak upper gradient. Marola has therefore posed a question, see Bj\"{o}rn and Bj\"{o}rn \cite[Open problem 2.13]{BjoBjo}, whether it is sufficient to assume that $f_k \to f$ in $L^p(\Pcal)$ and $g_k \to g$ in $L^p(\Pcal)$ as $k\to\infty$, where $p\in[1, \infty)$ and where $g_k$ and $g$ are minimal $L^p$-weak upper gradients of $f_k$ and $f$, respectively, for all $k$, in order to obtain that $f_k\to f$ in $N^{1,p}(\Pcal)\coloneq N^1\!L^p(\Pcal)$ as $k\to\infty$. If we study the same question in the setting of quasi-Banach function lattices, then such hypotheses certainly do not suffice in the following cases:
\begin{enumerate}
  \item The (quasi)norm of $X$ is not absolutely continuous, i.e., there exist a function $u\in X$ and a decreasing sequence of sets $E_n \to N$, where $\mu(N)=0$, such that $\|u \chi_{E_n}\|_X \not\to 0 = \| u \chi_N\|_X$ as $n\to\infty$. Typical examples of such spaces are $L^\infty$, the weak $L^p$ spaces, and the Marcinkiewicz spaces. Example  \ref{exa:conv_XxX_vs_N1X_a} below shows what kind of problems may arise in this setting.
  \item The space $\Pcal$ has infinite measure and the (quasi)norm of $X$ measures only the size of the peaks of a function whereas the ``rate of decay at infinity'' does not affect the value of the norm, e.g., $X=L^\infty + Z$, where $Z$ is an arbitrary (quasi)Banach function lattice, i.e., $\|u\|_X = \inf\{ \|v\|_{L^\infty} + \|w\|_Z: u=v+w\}$. In fact, these spaces may have an absolutely continuous norm. Example~\ref{exa:conv_XxX_vs_N1X_b} illustrates the situation for these function spaces.
\end{enumerate}
\end{rem}
%
%
\begin{exa}
\label{exa:conv_XxX_vs_N1X_a}
Let $X=L^\infty([0, 1])$. For $k\in\Nbb$, define 
\[
  u_k(x) = 
  \begin{cases} 
    \displaystyle \frac{2}{k} - x & \displaystyle \mbox{ for } 0\le x < \frac{1}{k}\,, \\[2ex]
    x & \displaystyle \mbox{ for } \frac{1}{k} \le x \le 1 \,.    
  \end{cases}
\]
Then, $u_k \to u$ in $L^\infty([0, 1])$ as $k\to\infty$, where $u(x) = x$. All functions $u_k$ as well as $u$ are $1$-Lipschitz with constant pointwise dilation, hence $g_{u_k}(x)=g_u (x) = 1$ for a.e. $x\in[0, 1]$. Thus, $g_{u_k} \to g_u$ in $L^\infty([0, 1])$ as $k\to\infty$. On the other hand, $u_k(x)-u(x) = 2(1/k-x)^+$, and hence $g_{u_k-u}(x) = 2 \chi_{[0,1/k)}(x)$, $x\in[0, 1]$. Therefore, $\|u_k - u\|_{N^{1,\infty}([0,1])} = 2/k + 2 \not\to 0$ as $k\to\infty$.
\end{exa}
%
%
\begin{exa}
\label{exa:conv_XxX_vs_N1X_b}
Let $X=L^1+L^\infty(\Rbb)$, where the function norm can be expressed by
\[
  \| f \|_{L^1 + L^\infty (\Rbb)} = \sup \Bigl\{\int_E |f|: E\subseteq \Rbb\mbox{ and } \lambda^1(E) \le 1\Bigr\} \quad \mbox{for $f\in\Mcal(\Rbb, \lambda^1)$}
\]
(see Bennett and Sharpley \cite[Theorem II.6.4 and Proposition~II.3.3]{BenSha}).

Let $\vphi$ be the $2$-periodic extension of the function $x\mapsto 1-|x-1|$, where $x\in [0, 2)$. For $k\in\Nbb$, define
\[
  u_k(x) = \begin{cases}
     2^{-j} \vphi( 2^j x) & \mbox{for } j-1 \le |x| < j,\mbox{ where } j\in\Nbb, j\neq k, \\
    -2^{-k} \vphi( 2^k x) & \mbox{for } k-1 \le |x| < k.
    \end{cases}
\]
Then, $u_k \to u$ in $L^\infty(\Rbb)$, and hence in $X$, as $k\to\infty$, where $u(x) = 2^{-j} \vphi( 2^j x)$ for $j-1 \le |x| < j$, $j\in\Nbb$. We might also observe that $u = |u_k|$ for any $k\in\Nbb$. All functions $u_k$ as well as $u$ are $1$-Lipschitz with constant pointwise dilation, hence $g_{u_k}(x)=g_u (x) = 1$ for a.e. $x\in\Rbb$. Thus, $g_{u_k} \to g_u$ in $X$ as $k\to\infty$. For $x\in\Rbb$, we have
\[
  u(x)-u_k(x) = 2^{1-k} \vphi(2^k x) \chi_{[k-1, k)}(|x|),
\]
whence $g_{u-u_k}(x) = 2 \chi_{[k-1, k)}(|x|)$. Thus, $\|u - u_k\|_{\NX} \ge \|g_{u-u_k} \|_X = 2 \not\to 0$ as $k\to\infty$.
\end{exa}
%
%


\begin{thebibliography}{00}
\bibitem{BenSha}
  Bennett, C., Sharpley, R.: \emph{Interpolation of Operators,}
  Pure and Applied Mathematics, vol. \textbf{129}, Academic Press, Orlando, FL, 1988.

\bibitem{BenLin}
  Benyamini, Y., Lindenstrauss, J.: \emph{Geometric Nonlinear Functional Analysis. Vol. 1,}
  AMS Colloquium Publications \textbf{48}, American Mathematical Society, Providence, RI, 2000.

\bibitem{BjoBjo}
  Bj\"{o}rn, A., Bj\"{o}rn, J.: \emph{Nonlinear Potential Theory on Metric Spaces,}
  EMS Tracts in Mathematics \textbf{17}, European Mathematical Society, Z\"{u}rich, 2011.

\bibitem{Bjo}
  Bj\"{o}rn, J.: Boundary continuity for quasiminimizers on metric spaces, \emph{Illinois J. Math.} \textbf{46} (2002), 383--403.

\bibitem{Bla}
  Blatter, J.: Reflexivity and the existence of best approximations, in \emph{Approximation theory, II (Austin, TX, 1976)}, pp. 299--301, Academic Press, New York, NY, 1976.
  
\bibitem{Che}
  Cheeger, J.: Differentiability of Lipschitz functions on metric measure spaces, \emph{Geom. Funct. Anal.}~\textbf{9} (1999), 428--517.

\bibitem{CosMir}
  Costea, \c{S}., Miranda, M.: Newtonian Lorentz metric spaces, to appear in \emph{Illinois J. Math.}, \texttt{arXiv:1104.3475}.

\bibitem{Fug}
  Fuglede, B.: Extremal length and functional completion, \emph{Acta Math.} \textbf{98} (1957), 171--219.

\bibitem{Haj2}
  Haj\l{}asz, P.: Sobolev spaces on metric-measure spaces, in \emph{Heat Kernels and Analysis on Manifolds, Graphs, and Metric Spaces (Paris, 2002),} Contemp. Math. \textbf{338}, pp. 173--218, Amer. Math. Soc., Providence, RI, 2003.

\bibitem{Hal}
  Halmos, P.\,R.: \emph{Measure Theory,} Graduate Texts in Mathematics, Vol. \textbf{18},
  Springer, New York, NY, 1950.

\bibitem{HalLux}
  Halperin, I., Luxemburg, W.\,A.\,J.: The Riesz--Fischer completeness theorem for function spaces and vector lattices, \emph{Trans. Roy. Soc. Canada. Sect. III} \textbf{50} (1956), 33--39.

\bibitem{Hei}
 Heinonen, J.: \emph{Lectures on Analysis on Metric Spaces,} Springer, New York, NY, 2001.
        
\bibitem{HeiKos0}
 Heinonen, J., Koskela, P.: From local to global in quasiconformal structures, \emph{Proc. Nat. Acad. Sci. U.S.A.} \textbf{93} (1996), 554--556.

\bibitem{HeiKos}
  Heinonen, J., Koskela, P.: Quasiconformal maps in metric spaces with controlled geometry, \emph{Acta Math.} \textbf{181} (1998), 1--61.

\bibitem{Jam}
  James, R.\,C.: Characterizations of reflexivity, \emph{Studia Math.} \textbf{23} (1964), 205--216.

\bibitem{KoMM}
  Koskela, P., MacManus, P.: Quasiconformal mappings and Sobolev spaces, \emph{Studia Math.} \textbf{131} (1998), 1--17.

\bibitem{LuxZaa}
  Luxemburg, W.\,A.\,J., Zaanen, A.\,C.: \emph{Riesz Spaces I,} North-Holland Mathematical Library, North-Holland, Amsterdam, 1971.

\bibitem{Mal}
  Mal\'{y}, L.: Newtonian spaces based on quasi-Banach function lattices, in \emph{Newtonian Spaces Based on Quasi-Banach Function Lattices}, Licentiate Thesis, pp. 11--35, Link\"{o}ping University, 2012.

\bibitem{Mat}
  Matkowski, J.: The converse of the Minkowski's inequality theorem and its generalization, \emph{Proc. Amer. Math. Soc.} \textbf{109} (1990), 663--675.

\bibitem{Moc3}
  Mocanu, M.: On the minimal weak upper gradient of a Banach--Sobolev function on a metric space, \emph{Sci. Stud. Res. Ser. Math. Inform.} \textbf{19} (2009), 119--129.

\bibitem{Mul}
  Mulholland, H.\,P.: Generalizations of Minkowski's inequality, \emph{Proc. London Math. Soc.} \textbf{51} (1950), 294--307.

\bibitem{Sha}
  Shanmugalingam, N.: Newtonian Spaces: An extension of Sobolev spaces to metric measure spaces, \emph{Rev. Mat. Iberoam.} \textbf{16} (2000), 243--279.

\bibitem{Sha2}
  Shanmugalingam, N.: Harmonic functions on metric spaces, \emph{Illinois J. Math} \textbf{45} (2001), 1021--1050.

\bibitem{Tuo}
  Tuominen, H.: Orlicz--Sobolev spaces on metric measure spaces, \emph{Ann. Acad. Sci. Fenn. Math. Dissertationes} \textbf{135} (2004).

\bibitem{Zaa}
  Zaanen, A.\,C.: \emph{Riesz Spaces II,} North-Holland Mathematical Library \textbf{30}, North-Holland, Amsterdam, 1983.

\bibitem{Zie}
  Ziemer, W.\,P.: \emph{Weakly Differentiable Functions,} Springer, New York, NY, 1989.

\end{thebibliography}
\end{document}